\documentclass[onefignum,onetabnum]{siamart171218}
\usepackage{amsmath,amssymb}
\usepackage{color}
\newcommand{\trx}{^{\mbox{\tiny\sf\bfseries T}}}

\usepackage{verbatim}
\headers{Mean-square contractivity of stochastic $\theta$-methods}{R. D'Ambrosio, S. Di Giovacchino}

\title{Mean-square contractivity of stochastic $\theta$-methods\thanks{The authors are member of the INdAM Research group GNCS. The work is supported by GNCS-Indam project and by PRIN2017-MIUR project.}}

\author{Raffaele D'Ambrosio\thanks{Department of Engineering and Computer Science and Mathematics, University of L'Aquila, Italy
  (\email{raffaele.dambrosio@univaq.it}).}
\and  Stefano Di Giovacchino\thanks{Department of Engineering and Computer Science and Mathematics, University of L'Aquila, Italy
  (\email{stefano.digiovacchino@graduate.univaq.it}).}}

\begin{document}

\maketitle

% REQUIRED
\begin{abstract}
The paper is focused on the nonlinear stability analysis of stochastic $\theta$-methods. In particular, we consider nonlinear stochastic differential equations such that the mean-square deviation between two solutions exponentially decays, i.e., a mean-square contractive behaviour is visible along the stochastic dynamics. We aim to make the same property visible also along the numerical dynamics generated by stochastic $\theta$-methods: this issue is translated into sharp stepsize restrictions depending on parameters of the problem, here accurately estimated. A selection of numerical tests confirming the effectiveness of the analysis and its sharpness is also provided.
\end{abstract}

% REQUIRED

% REQUIRED

\setcounter{equation}{0}

\section{Introduction}
We consider a nonlinear system of stochastic differential equations (SDEs) of It\^{o} type, assuming the form
\begin{equation}\label{equation}
\left\{\begin{aligned}
dX(t)&= f(X(t))dt + g(X(t))dW(t), \quad t \in [0,T],\\
X(0)&= X_0,
\end{aligned}\right.
\end{equation}
where $f: \mathbb{R}^n \rightarrow \mathbb{R}^n$, $g : \mathbb{R}^n \rightarrow \mathbb{R}^{n \times m}$ and
$W(t)$ is an $m$-dimensional Wiener process. For theoretical results on the existence and uniqueness of solutions to \eqref{equation}, we refer to the monograph \cite{Kloeden1992}. Moreover, in the sequel, we assume that \eqref{equation} is commutative.

We focus our attention on providing a nonlinear stability analysis to the general classes of stochastic $\theta$-methods for \eqref{equation} that,  with reference to the discretized domain $\mathcal{I}_{\Delta t}=\{t_n=n\Delta T, \ n=0,1,\ldots,N, \ N=T/\Delta t\}$, assume the following forms: 

\begin{align}
    X_{n+1} &= X_n + (1-\theta)\Delta t f (X_n) +  \theta \Delta t f(X_{n+1}) + g(X_n)\Delta W_n,\label{teta}\\
    X_{n+1} &= X_n + (1-\theta)\Delta t f (X_n) + \theta \Delta t f(X_{n+1}) + \sum_{j=1}^{m} g^j(X_n)\Delta W_n^j \label{teta_m}\\
    & + \frac{1}{2} \sum_{j=1}^{m } L^j g^j(X_n)((\Delta W_n^j)^2 - \Delta t)+\frac{1}{2} \sum_{\substack{j_1,j_2=1\\j_1 \ne j_2}}^{m}L^{j_1}g^{j_2}(X_n)\Delta W_n^{j_1} \Delta W_n^{j_2}, \notag
    \end{align}
where $\theta \in [0,1]$, $X_n$ is the approximate value for $X(t_n)$, the discretized Wiener increment $\Delta W_n$ is distributed as a gaussian random variable with zero mean and variance $\Delta t$, the operator $L^j$ is defined as
\begin{equation}
\notag
L^j = \sum_{k=1}^{n} g^{k,j} \frac{\partial}{\partial x^k}, \ \ \ j=1,...,m,
\end{equation}
where $g^j(X_n)$ is the $j-$th column of the matrix $g(X_n)$ and $\Delta W_n^j$ the $j-$th element of vector $\Delta W_n$. 
We refer to \eqref{teta} as $\theta$-{\em Maruyama} method and to \eqref{teta_m} as $\theta$-{\em Milstein} method in its componentwise form.
We note that, if $m=1$, \eqref{teta_m} reduces to the form 
\begin{equation}
\notag
\begin{aligned}
    X_{n+1} = X_n &+ (1-\theta)\Delta t f (X_n) + \theta \Delta t f(X_{n+1}) +  g(X_n)\Delta W_n\\
    & + \frac{1}{2} g(X_n) g'(X_n)(\Delta W_n^2 - \Delta t).
      \end{aligned}
\end{equation}

The stability analysis of $\vartheta$-methods has been given in \cite{buck11,r4} with respect to linear test problems, both scalar and vector-valued. The investigation led to conditions according to which the mean-square and asymptotic behaviours of the solutions to such linear problems are inherited also along the discretized counterpart provided by above $\vartheta$-methods. 

This paper is instead focused on providing a nonlinear stability analysis for $\vartheta$-methods \eqref{teta} and \eqref{teta_m}, with the aim to numerically inherit relevant properties of nonlinear problems along their discretizations. The discussion is motivated by some contributions on the so-called exponential stability properties of nonlinear SDEs, contained in \cite{r6,r7} and here briefly summarized in the following result.
\begin{theorem}\label{thExp}
For a given nonlinear SDE \eqref{equation}, let us assume the following properties for the drift $f$ and the diffusion $g$, by denoting with $|\cdot|$ both the Euclidean norm in $\mathbb{R}^n$ and the trace (or Frobenius) norm in $\mathbb{R}^{n\times m}$:
\begin{enumerate}
\item[(i)] $f, g \in \mathcal{C}^1(\mathbb{R}^n)$;
\item[(ii)] f satisfies a one-side Lipschitz condition, i.e. there exists $\mu \in \mathbb{R}$ such that
\begin{equation}\label{onesided}
<x-y, f(x)-f(y)> \le \mu \left| x-y \right|^2 , \ \ \ \forall x,y \in \mathbb{R}^n;
\end{equation}
\item[(iii)] g is a globally Lipschitz function, i.e. there exists $L > 0$ such that
\begin{equation}
\left| g(x)-g(y) \right|^2 \le L \left| x- y \right|^2 \ \ \ \forall x,y \in \mathbb{R}^n.
\end{equation}
\end{enumerate}
Then, any two solutions $X(t)$ and $Y(t)$ of (\ref{equation}), with $\mathbb{E}\left|X_0 \right|^2 < \infty$ and
$\mathbb{E}\left|Y_0 \right|^2 < \infty$,  satisfy
\begin{equation}
\label{exp}
\mathbb{E}\left|X(t) - Y(t)\right|^2 \le \mathbb{E}\left|X_0-Y_0 \right|^2 e^{\alpha t} , 
\end{equation}
where $\alpha = 2\mu + L$.
\end{theorem}

The inequality \eqref{exp} is denoted as {\em exponential mean-square stability inequality} for \eqref{equation}. An eventual negative sign on the parameter $\alpha$ appearing in the stability inequality \eqref{exp} allows to infer an exponential decay of the mean-square deviation between two solutions of a given SDE \eqref{equation}. Motivated by an analog property of deterministic differential equations (see, for instance, \cite{hawa} and references therein) we then introduce the following definition. 

\begin{definition}\label{def1}
A nonlinear SDE (\ref{equation}) whose solutions satisfy the exponential stability inequality (\ref{exp}) with $\alpha < 0$ is said to generate exponential mean-square contractive solutions.
\end{definition}

We observe that, when $g$ is identically zero in \eqref{equation}, Definition \ref{def1} recovers the deterministic condition $\mu<0$ that guarantees the contractive behaviour of the solutions to the corresponding deterministic problem. The discretization of deterministic differential equations with one-sided Lipschitz vector field with negative one-sided Lipschitz constant led to the notion of G-stability of numerical methods, introduced by G. Dahlquist in \cite{dahl76}. 

Here we aim to provide the numerical counterpart of exponential mean-square contractivity, that is certainly a relevant property to be inherited also by the discretized problem, since it ensures a long-term damping of the error along the numerical solutions. In particular, we aim to prove that the stability inequality \eqref{exp} is translated into a restriction on the stepsize employed in the numerical discretization, here sharply estimated. The provided inequalities characterizing the numerical methods depend on parameters that are also accurately estimated, in order to make the corresponding restrictions on the stepsize fully computable. 

The paper is organized as follows: Section 2 briefly recalls the main results regarding the linear stability properties of stochastic $\theta$-methods; Section 3 provides exponential mean-square stability inequalities for the $\theta$-methods \eqref{teta} and \eqref{teta_m}, giving the numerical counterpart of \eqref{exp}; in Section 4 we give a notion of mean-square contractivity for the numerical solutions computed by \eqref{teta} and \eqref{teta_m}, which is here translated into stepsize restrictions depending on parameters which are here estimated; Section 5 shows the numerical evidence on a selection of nonlinear problems \eqref{equation}, confirming the sharpness of the provided estimates; some conclusions are presented in Section 6. 

\section{Linear stability of stochastic $\theta$-methods}
It is worth recalling the main results regarding the linear stability properties of stochastic $\theta$-methods \eqref{teta} and \eqref{teta_m}, according to \cite{buck11,r4}. Indeed, the stepsize restrictions we present in the next sections in order to ensure the conservation of the exponential mean-square contractivity along numerical solutions clearly have to be compatible with the linear stability properties of the corresponding method.  

The linear stability analysis for the discretization of SDEs \eqref{equation}, as well known for instance from \cite{r4,h01,r5}, 
is performed with respect to the linear scalar problem
\begin{equation}\label{linearEquation}
\left\{\begin{aligned}
dX(t)&= \lambda X(t)dt + \mu X(t)dW(t), \quad t \in [0,T],\\
X(0)&= X_0,
\end{aligned}\right.
\end{equation}
with $\lambda, \mu \in \mathbb{C}$. The following definition occurs (see, for instance \cite{r4,h01}).
\begin{definition} 
The solution $X(t)$ of (\ref{linearEquation}) is mean-square stable if 
$$\lim_{t \rightarrow \infty} \mathbb{E}\left| X(t) \right|^2 =0.$$
\end{definition}
As proved in \cite{r4,h01,r5}, the solution $X(t)$ to (\ref{linearEquation}) is mean-square stable if and only if
\begin{equation}
\textrm{Re}(\lambda) + \frac{1}{2}\left| \mu \right|^2 < 0.
\end{equation}

The numerical counterpart of above arguments is provided in the following definition \cite{r4,h01,r5}.
\begin{definition}
\label{meanSqSt}
The numerical solution $X_n$ of (\ref{linearEquation}) is mean-square stable if
$$\lim_{n\rightarrow \infty} \mathbb{E}\left| X_n \right|^2 =0.$$
\end{definition}
Correspondingly, according to \cite{r4}, the stochastic $\theta$-Maruyama method \eqref{teta} is mean-square stable if and only if
\begin{equation}\label{mean-squareCond}
\frac{\left| 1 + (1-\theta) \Delta t \lambda \right|^2 + \Delta t \left| \mu \right|^2}{\left| 1-\theta \Delta t \lambda \right|^2}<1.
\end{equation}

Let us provide an analogous condition for the $\theta$-Milstein method (\ref{teta_m}). To this purpose, we apply the method (\ref{teta_m}) to (\ref{linearEquation}), obtaining
\begin{equation}
\label{myE}
X_{n+1} = \bigg[ \frac{1- (1-\theta)\Delta t \lambda + \mu + \frac{1}{2} \mu^2 (\Delta W_n^2 - \Delta t) }{1-\theta \Delta t \lambda} \bigg] X_n.
\end{equation}
Squaring and passing to the expectation leads to
$$\mathbb{E} | X_{n+1} |^2 =\bigg[  \beta^2 + \frac{\beta \mu^2 \Delta t}{1-\theta \Delta t \lambda} + \frac{\mu^2 \Delta t + \frac{3}{4} \mu^4 \Delta t^2}
 {(1-\theta \Delta t \lambda)^2} \bigg] \mathbb{E} | X_n |^2,$$
with
$$\beta = \frac{1+ (1-\theta)\Delta t \lambda - \frac{1}{2}\mu^2 \Delta t}{1-\theta \Delta t \lambda}.$$
Therefore, according to Definition \ref{meanSqSt}, the $\theta$-Milstein method (\ref{teta_m}) is mean-square stable if and only if
\begin{equation}\label{teta_mMeanSquareSt}
 \bigg | \beta^2 + \frac{\beta \mu^2 \Delta t}{1-\theta \Delta t \lambda} + \frac{\mu^2 \Delta t + \frac{3}{4} \mu^4 \Delta t^2}
 {(1-\theta \Delta t \lambda)^2} \bigg | < 1.
\end{equation}

In the remainder, we check that all the values of $\theta$ and $\Delta t$ leading to an exponential mean-square contractive behaviour of the numerical solution to \eqref{equation} computed by \eqref{teta} or \eqref{teta_m} fulfill the constraints given by \eqref{mean-squareCond} and \eqref{teta_mMeanSquareSt}, respectively. 

\section{Exponential mean-square stability inequalities} We aim to provide the numerical counterpart of \eqref{exp}, i.e., we develop an analogous exponential mean-square stability inequality for the numerical discretization of \eqref{equation} with the $\theta$-methods \eqref{teta} and \eqref{teta_m}, under the assumptions of Theorem \ref{thExp} . The following technical lemma (see \cite{r7}) is useful in the remainder. 

\begin{lemma}\label{lem}
Under the assumptions (i)--(iii) given in Theorem \ref{thExp}, for any $h>0$ and $b_1, b_2 \in \mathbb{R}^n$ and , there exist unique $a_1, a_2 \in \mathbb{R}^n$ solutions of the implicit equations
$$a_i - h f(a_i) = b_i, \quad i = 1,2,$$
satisfying the inequality
\begin{equation}
\notag
(1-2h\mu) \left| a_1-a_2 \right|^2 \le \left| b_1-b_2 \right|^2.
\end{equation}
\end{lemma}

\subsection{Exponential mean-square stability of $\theta$-Maruyama methods} 
The following result provides the counterpart of \eqref{exp} for the numerical discretization of \eqref{equation} with the $\theta$-Maruyama method \eqref{teta}.

\begin{theorem}\label{THM}
Under the assumptions (i)--(iii) given in Theorem \ref{thExp}, any two numerical solutions $X_n$ and $Y_n$, $n\geq 0$, computed by applying the $\theta$-Maruyama method \eqref{teta} to \eqref{equation} with initial values such that $\mathbb{E} \left| X_0 \right|^2 < \infty$ and $\mathbb{E} \left| Y_0 \right|^2 < \infty$, satisfy the inequality
\begin{equation}
\label{contCond}
\mathbb{E}\left| X_{n} - Y_{n} \right|^2 \le \mathbb{E}\left| X_{0} - Y_{0} \right|^2 \textrm{e}^ { \nu(\theta,\Delta t) t_n} ,
\end{equation}
where 
\begin{equation}\label{nu_mar}
\nu(\theta,\Delta t) = \frac{1}{\Delta t} \ln{\beta(\theta, \Delta t)} 
\end{equation}
and
\begin{equation}\label{beta_mar}
\beta(\theta,\Delta t) = 1+ \frac{\alpha + (1-\theta)^2 M \Delta t }{1 - 2\theta\mu \Delta t} \Delta t,
\end{equation}
with
\begin{equation}\label{M}
M=\sup_{t\in[0,T]}\mathbb{E}|f'(X(t))|^2.
\end{equation}
\end{theorem}
\begin{proof} 
Since $X_{n+1}$ and $Y_{n+1}$ satisfy the implicit equations
$$\begin{aligned}
X_{n+1} &= X_n + (1-\theta)\Delta t f (X_n) +  \theta \Delta t f(X_{n+1}) + g(X_n)\Delta W_n,\\
Y_{n+1} &= Y_n + (1-\theta)\Delta t f (Y_n) +  \theta \Delta t f(Y_{n+1}) + g(Y_n)\Delta W_n,
\end{aligned}$$
according to Lemma \ref{lem} we obtain
\begin{equation}\label{tmp1}
(1 - 2\theta\mu \Delta t) |X_{n+1} - Y_{n+1} |^2 \le \bigl|(X_n - Y_n) + (1 - \theta)\Delta t \Delta f_n + \Delta g_n\Delta W_n \bigr|^2,
\end{equation}
where $\Delta f_n =f(X_n)-f(Y_n)$ and $\Delta g_n=g(X_n)-g(Y_n)$. The right-hand side of \eqref{tmp1} is then bounded by
$$\begin{aligned}
&|X_n-Y_n|^2 + (1-\theta)^2 \Delta t ^2 |\Delta f_n|^2+ |\Delta g_n \Delta W_n|^2+ 2(1-\theta)\Delta t <X_n-Y_n,\Delta f_n>\\
&\qquad\qquad \ \ + 2<X_n-Y_n,\Delta g_n \Delta W_n>+ 2 (1-\theta)\Delta t <\Delta f_n,\Delta g_n \Delta W_n>.
\end{aligned}$$
By applying the assumptions (i)--(iii) of Theorem \ref{thExp}, we obtain
$$\begin{aligned}
(1 - 2\theta\mu \Delta t) |X_{n+1} - Y_{n+1} |^2 &\le (1+L |\Delta W_n|^2+ 2(1-\theta)\Delta t \mu)|X_n-Y_n|^2\\  
&+(1-\theta)^2 \Delta t ^2 |\Delta f_n|^2 + 2<X_n-Y_n,\Delta g_n \Delta W_n>\\
&+ 2 (1-\theta)\Delta t <\Delta f_n,\Delta g_n \Delta W_n>.
\end{aligned}$$
Passing to the expectations leads to
$$\mathbb{E}|X_{n+1} - Y_{n+1} |^2 \le \beta(\theta,\Delta t) \mathbb{E} \left| X_n - Y_n \right|^2.$$
Since
\begin{equation}
\notag
\mathbb{E}|X_{n+1} - Y_{n+1} |^2 \le\beta(\theta,\Delta t) ^{n+1} \mathbb{E}|X_{0} - Y_{0} |^2,
\end{equation}
the thesis holds true.
\end{proof}

According to Theorem \ref{THM}, the $\theta$-Maruyama method \eqref{teta} satisfies the exponential mean-square stability inequality \eqref{contCond} with argument $\nu(\theta,\Delta t)$ of the exponential given by \eqref{nu_mar}, when applied to the SDE \eqref{equation} satisfying the inequality \eqref{exp} with parameter $\alpha=2\mu+L$. Let us now provide an estimate for the error $|\nu(\theta,\Delta t)-\alpha|$.

\begin{theorem}\label{nuEst}
Under the same assumptions of Theorem \ref{THM}, for any fixed value of $\theta\in[0,1]$, we have
\begin{equation}
\label{alfaEst}
|\nu(\theta,\Delta t)-\alpha|=\mathcal{O}(\Delta t).
\end{equation}
\end{theorem}
\begin{proof}
By expanding $\nu(\theta,\Delta t)$ in \eqref{nu_mar} in power series of $\Delta t$, we obtain
$$\nu (\theta,\Delta t) = \alpha + \left(2\alpha\mu\theta + (1-\theta)^2M - \frac{\alpha^2}{2} \right)\Delta t + O(\Delta t^2),$$
leading to the thesis.
\end{proof}

\subsection{Exponential mean-square stability of $\theta$-Milstein methods}
The following result is focused on the counterpart of \eqref{exp} for the numerical discretization of \eqref{equation} with the $\theta$-Milstein method \eqref{teta_m}.
\begin{theorem}\label{THM_m}
Under the assumptions (i)--(iii) given in Theorem \ref{thExp}, any two numerical solutions $X_n$ and $Y_n$, $n\geq 0$, computed by applying the $\theta$-Milstein method \eqref{teta_m} to \eqref{equation} with initial values such that $\mathbb{E} \left| X_0 \right|^2 < \infty$ and $\mathbb{E} \left| Y_0 \right|^2 < \infty$, satisfy the inequality
\begin{equation}
\label{contCond2}
\mathbb{E}\left| X_{n} - Y_{n} \right|^2 \le \mathbb{E}\left| X_{0} - Y_{0} \right|^2 e^ { \epsilon(\theta,\Delta t) t_n} ,
\end{equation}
where 
\begin{equation}\label{eps_mil}
\epsilon(\theta,\Delta t) = \frac{1}{\Delta t} \ln{\gamma(\theta, \Delta t)} 
\end{equation}
and
\begin{equation}\label{gamma_mil}
\gamma(\theta,\Delta t) =\beta(\theta,\Delta t) +\frac{3 \widetilde{M}\Delta t^2}{4(1 - 2\theta\mu \Delta t)},
\end{equation}
with $\widetilde{M}$ defined as
\begin{equation}\label{Mtilde}
\widetilde{M}=\sum_{i,j=1}^{m}\sum_{k,l=1}^{n} \widetilde{M}_{i,j}^{k,l},
\end{equation}
where
\begin{equation}
\notag
\widetilde{M}^{k,l}_{i,j} = \sup_{t\in[0,T]}\frac{\mathbb{E}\bigg( h_{i,j}^{k,l}(X(t),Y(t))\bigg)}{\mathbb{E}|X(t)-Y(t)|^2},
\end{equation}
being
\begin{equation}
\notag
\begin{aligned}
h^{k,l}_{i,j}(X(t),Y(t)) =&< g^{k,i}(X(t))\frac{\partial}{\partial x^k}g^j(X(t))-g^{k,i}(Y(t))\frac{\partial}{\partial y^k}g^j(Y(t)),\\
&g^{l,i}(X(t))\frac{\partial}{\partial x^l}g^j(X(t))-g^{l,i}(Y(t))\frac{\partial}{\partial y^l}g^j(Y(t))>, 
\end{aligned}
\end{equation}
$i,j=1,\ldots,m$, $k,l=1,\ldots,n$.
 \end{theorem}
 \begin{proof}
 First, we show the proof for $m=1$. 
Since $X_{n+1}$ and $Y_{n+1}$ satisfy the implicit equations
$$\begin{aligned}
X_{n+1} &= X_{n} + (1-\theta)\Delta t f(X_n) + \theta \Delta t f(X_{n+1}) + g(X_n)\Delta W_n+ \frac{1}{2} h (X_n) (\Delta W_n^2-\Delta t),\\
Y_{n+1} &= Y_{n} + (1-\theta)\Delta t f(Y_n) + \theta \Delta t f(Y_{n+1}) + g(Y_n)\Delta W_n+ \frac{1}{2} h(Y_n) (\Delta W_n^2-\Delta t),
\end{aligned}$$
where $h(x) = g(x)g'(g),$
Lemma \ref{lem} leads to
$$\begin{aligned}
(1 - 2\theta\mu \Delta t) |X_{n+1} - Y_{n+1} |^2 \le \biggl|(X_n - Y_n) &+ (1 - \theta)\Delta t \Delta f_n+ \Delta g_n\Delta W_n \biggr. \\ 
& \left. +\frac12 \Delta h_n(\Delta W_n^2-\Delta t)\right|^2.
 \end{aligned}$$
By proceeding as in Theorem \ref{THM}, the thesis holds true for $m=1$. The general case $m>1$ holds true as direct generalization of the previous one.\\
\end{proof}

According to Theorem \ref{THM_m}, the $\theta$-Milstein method \eqref{teta_m} satisfies the exponential mean-square stability inequality \eqref{contCond} with argument $\epsilon(\theta,\Delta t)$ of the exponential given by \eqref{eps_mil}, when applied to the SDE \eqref{equation} satisfying the inequality \eqref{exp} with parameter $\alpha=2\mu+L$. As in the previous section, let us now provide an estimate for the error $|\epsilon(\theta,\Delta t)-\alpha|$.

\begin{theorem}\label{epEst}
Under the same assumptions of Theorem \ref{THM_m}, for any fixed value of $\theta\in[0,1]$, we have
\begin{equation}\label{alfaEst2}
|\epsilon(\theta,\Delta t)-\alpha|=\mathcal{O}(\Delta t).
\end{equation}
\end{theorem}
\begin{proof}
By expanding $\epsilon(\theta,\Delta t)$ in \eqref{eps_mil} in power series of $\Delta t$, we obtain
$$\epsilon (\theta,\Delta t) = \alpha + \left(2\alpha\mu\theta  + (1-\theta)^2M - \frac{\alpha^2}{2}+\frac{3\widetilde{M}}4 \right)\Delta t + O(\Delta t^2),$$
leading to the thesis.
\end{proof}

\section{Mean-square contractivity}
According to Definition \ref{def1}, the nonlinear sto\-chastic system \eqref{equation} generates mean-square contractive solutions if $\alpha < 0$ in \eqref{exp}. A natural counterpart of this definition for a stochastic $\theta$-method is now given as follows. 
\begin{definition}\label{contDef}
Consider a nonlinear stochastic differential equation (\ref{equation}) satisfying assumptions (i)--(iii) given in Theorem \ref{thExp} and let $X_n$ and $Y_n$, $n\geq 0$, be two numerical solutions of  (\ref{equation}) computed by the $\theta$-methods (\ref{teta}) or (\ref{teta_m}). Then, the applied method is said to be generate mean-square contractive numerical solutions in a region $\mathcal{R}\subseteq\mathbb{R}^{+}$ if, for a fixed $\theta\in[0,1]$,
$$\nu(\theta,\Delta t)<0, \quad \forall\Delta t\in \mathcal{R}$$
for \eqref{teta}, being $\nu(\theta,\Delta t)$ the parameter in \eqref{contCond}, or
$$\epsilon(\theta,\Delta t)<0, \quad \forall\Delta t\in \mathcal{R}$$
for \eqref{teta_m}, where $\epsilon(\theta,\Delta t)$ is the parameter in \eqref{contCond2}.
\end{definition} 
\begin{definition}\label{uncondContr}
A stochastic $\theta$-method (\ref{teta}) or (\ref{teta_m}) is said unconditionally mean-square contractive if, for a given $\theta\in[0,1]$, $\mathcal{R}={\mathbb R}^+$.
\end{definition} 

As regards the $\theta$-Maruyama method \eqref{teta}, according to Definition \ref{contDef}, mean-square contractive numerical solutions are generated if 
$$0 < \beta(\theta,\Delta t)<1,$$
for any $\Delta t\in \mathcal{R}$, i.e.
\begin{equation}\label{cc1}
\mathcal{R}=\begin{cases}
\left(0,\displaystyle\frac{|\alpha|}{(1-\theta)^2 M}\right), &  \theta < 1,\\[3mm]
\mathbb{R}^+, &  \theta = 1. 
\end{cases}
\end{equation}

As a consequence, we have proved the following result for the $\theta$-Maruyama method with $\theta=1$, i.e., for the implicit Euler-Maruyama method
\begin{equation}\label{implEM}
X_{n+1} = X_n + \Delta t f(X_{n+1}) + g(X_n)\Delta W_n.
\end{equation}

\begin{theorem}\label{thUC}
For a given a nonlinear problem \eqref{equation} satisfying the assumptions (i)--(iii) given in Theorem \ref{thExp}, the implicit Euler-Maruyama method \eqref{implEM} is unconditionally mean-square contractive.
\end{theorem}

In other terms, the stochastic perturbation \eqref{implEM} of the deterministic implicit Euler method preserves its unconditional contractivity property \cite{hawa}.

In analogous way, as regards the $\theta$-Milstein method \eqref{teta_m}, Definition \ref{contDef} leads to 
$$0 < \gamma(\theta,\Delta t)<1,$$
for any $\Delta t\in \mathcal{R}$, i.e.
\begin{equation}\label{cc2}
\mathcal{R}=\begin{cases}
\left(0,\displaystyle\frac{4|\alpha|}{4(1-\theta)^2 M +3 \widetilde{M}}\right), &  \theta < 1,\\[5mm]
\left(0,\displaystyle\frac{4|\alpha|}{3\widetilde{M}}\right), &  \theta = 1. 
\end{cases}
\end{equation}

The computation of the regions $\mathcal{R}$ in \eqref{cc1} and \eqref{cc2} relies on the knowledge of the Lipschitz constant $L$ to the diffusion of \eqref{equation}, the one-sided Lipschitz constant $\mu$ of the drift, the constants $M$ and $\widetilde{M}$ defined by \eqref{M} and \eqref{Mtilde}, respectively. The estimation of the parameters $L$ and $\mu$ is typically required in global optimization algorithms, therefore we adopt a similar estimation strategy (see \cite{r8}) to make the region $\mathcal{R}$ fully computable. 

\newpage

\begin{center}
{\bf Algorithm 1: estimation of the Lipschitz constant $L$}
\end{center}

\bigskip

\textbf{Step 1}. We perform $P$ paths of the $\theta$-methods \eqref{teta} or \eqref{teta_m} and denote by $X_n^{i,j}$ the $i$-th component of the $j$-th realization of the solution $X_n$, $i=1,2,\ldots,d$, $j=1,2,\ldots,P$. Then, we compute 
\begin{equation}
a_i= \min_{j=1,\ldots,P}\min_{t_n\in\mathcal{I}_{\Delta t}}{X^{i,j}_n},  \quad \quad b_i = \max_{j=1,\ldots,P}\max_{t_n\in \mathcal{I}_{\Delta t}}{X^{i,j}_n},
\end{equation}
$i=1,2,\ldots,d$.\\

\textbf{Step 2}. We generate $Q$ couples of vectors 
$$x_k=\left[x_k^{1}, \ x_k^{2}, \ \ldots, \ x_k^{d}\right]\trx, \quad y_k=\left[y_k^{1}, \ y_k^{2}, \ \ldots, \ y_k^{d}\right]\trx,$$ 
with $k=1,2,\ldots Q$, such that $(x_k^i,y_k^i)$ is uniformly distributed in $[a_i,b_i] \times[a_i,b_i]$, $i=1,2,\ldots,d$. \\

\textbf{Step 3}. We compute
\begin{equation}
\notag
s_k = \frac{|g(x_k)-g(y_k)|^2}{|x_k - y_k |^2}, \quad\quad k = 1,2,\ldots,Q.\\
\end{equation}

\textbf{Step 4}. We assume as estimate of $L$ the value of $\max\{s_1,...,s_Q\}$.\\

For a detailed accuracy analysis of the algorithm, we refer to \cite{r8}. An analogous algorithm for the estimate of the one-sided Lipschitz constant $\mu$ is obtained in a similar way.

\bigskip

\begin{center}
{\bf Algorithm 2: estimation of the one-sided Lipschitz constant $\mu$}
\end{center}

\bigskip

\textbf{Step 1}. See Step 1 of Algorithm 1.\\

\textbf{Step 2}. See Step 2 of Algorithm 1.\\

\textbf{Step 3}. We compute
\begin{equation}
\notag
s_k = \frac{<x_k-y_k,f(x_k)-f(y_k)>}{|x_k - y_k |^2}, \quad\quad k = 1,2,\ldots,Q.\\
\end{equation}

\textbf{Step 4}. We assume as estimate of $\mu$ the value of $\min\{s_1,...,s_Q\}$.\\

Clearly the estimates of $M$ in \eqref{M} and $\widetilde{M}$ in \eqref{Mtilde} is straightforward from their definitions, once $P$ realizations of the numerical solution are computed.

\section{Numerical experiments} In this section, we present the numerical evidence arising from the application of the $\theta$-Maruyama \eqref{teta} and the $\theta$-Milstein \eqref{teta_m} methods to a selection of nonlinear problems generating mean-square contractive solutions according to Definition \ref{def1}. We confirm the sharpness of the estimates provided in Section 4 for the stepsize $\Delta t$ in order to generate mean-square contractive numerical solutions according to Definition \ref{contDef}. The expected values computed in the remainder of this section always rely on the numerical solutions over $P=2000$ paths.

\medskip

{\em Problem 1}.  We consider the scalar SDE \eqref{equation} with 
\begin{equation}
\notag
f(X(t)) = -4X(t)-X(t)^3, \quad \quad g(X(t)) = X(t)
\end{equation}
and initial data $X_0=1$ and $Y_0=0$, used as test example in \cite{r6}. For this problem the constants $L$ and $\mu$ are given by $L=1$ and $\mu=-4$, so $\alpha=-7$. Then, according to Theorem \ref{thExp}, this problem generates mean-square contractive  solutions. Moreover, the values of $M$ in \eqref{M} and $\widetilde{M}$ in \eqref{Mtilde} are 16 and 1, respectively. We consider the following $A$-stable methods \cite{r4}: 

\medskip

\begin{itemize}
\item the stochastic trapezoidal methods, i.e., the $\theta$-Maruyama methods \eqref{teta} with $\theta=1/2$. In this case (\ref{cc1}) yields 
$$\mathcal{R}=\left(0, \ \frac{7}{4}\right).$$
The corresponding estimate on $\Delta t$ is confirmed in Figure \ref{mar_pol_05}, where the time-evolution of the mean-square deviation $\mathbb{E}|X_n-Y_n|^2$ in logarithmic scale is depicted for various values of $\Delta t$. It is visible that, the more $\Delta t$ decreases, the more the numerical slope $\nu(\frac12,\Delta t)$ in \eqref{contCond} tends to the exact slope $\alpha$ in \eqref{exp}. For values of $\Delta t>\frac{7}{4}$, the mean-square deviation does not exponentially decay;

\begin{figure}
\centering
\scalebox{0.5}{\includegraphics{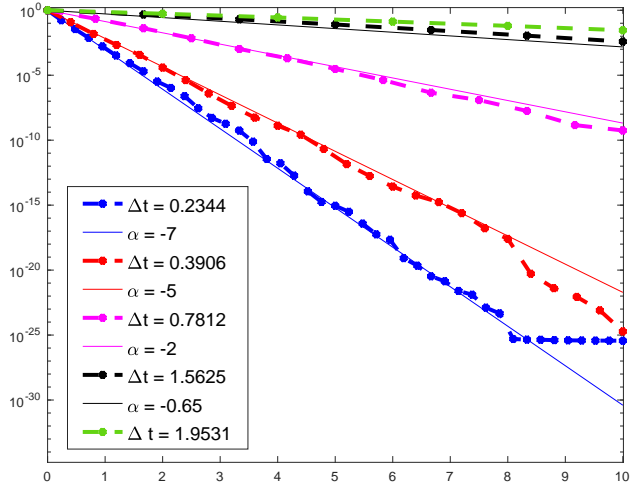}}
          \vspace{-0.4cm} \caption{Mean-square deviations over 2000 paths for the stochastic trapezoidal method applied to Problem 1.}\label{mar_pol_05}
\end{figure}

\item the stochastic implicit Euler \eqref{implEM}, that is unconditionally mean-square contractive, according to Theorem \ref{thUC}. The behaviour depicted in Figure \ref{mar_pol_1} confirms the theoretical result on the unconditional contractivity of \eqref{implEM}. Indeed, the mean-square deviation is always exponentially decaying and its slope tends to the exact slope as $\Delta t$ decreases;
 \begin{figure}
\centering
         \scalebox{0.5}{\includegraphics{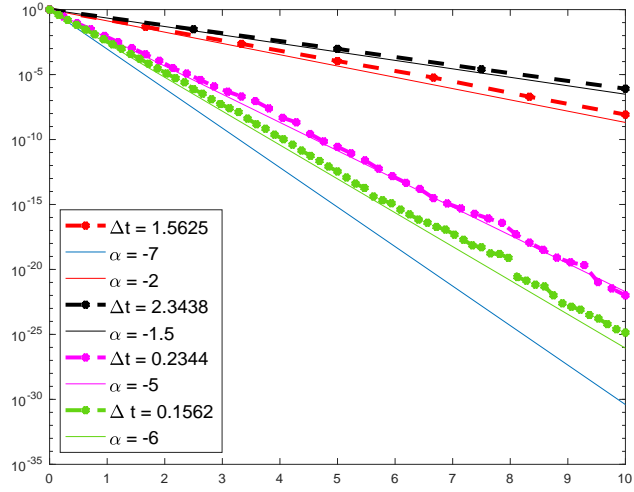}}
          \vspace{-0.4cm}\caption{Mean-square deviations over 2000 paths for the stochastic implicit Euler method \eqref{implEM}, applied to Problem 1.}\label{mar_pol_1}
\end{figure}
\item the $\theta$-Milstein method \eqref{teta_m} with $\theta=1/2$.  For this method, (\ref{cc2}) leads to 
$$\mathcal{R}=\left(0, \ \frac{14}9\right).$$
Also in this case, as shown in Figure \ref{mil_pol_05}, the theoretical estimate of $\Delta t$ is confirmed by the numerical evidence. As already proved in Theorem \ref{epEst}, the numerical slope $\epsilon(\frac12,\Delta t)$ in \eqref{contCond2} tends to the exact slope $\alpha$ in \eqref{exp}. For values of $\Delta t>\frac{14}{9}$, the mean-square deviation does not exponentially decay.
 \begin{figure}
\centering
         \scalebox{0.5}{\includegraphics{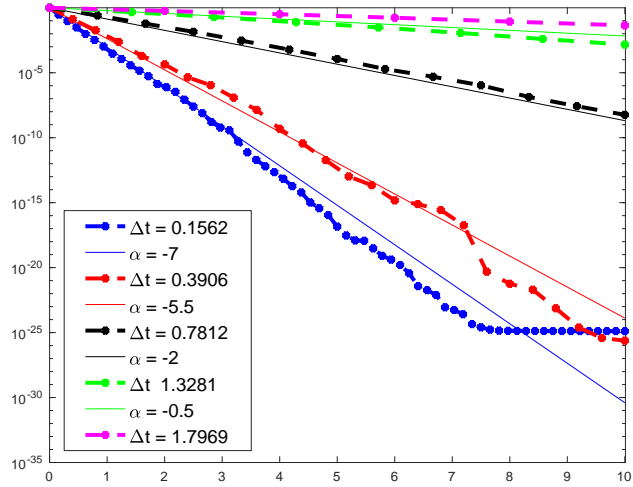}}
          \vspace{-0.4cm}\caption{Mean-square deviations over 2000 paths for the $\theta$-Milstein method \eqref{teta_m} with $\theta=1/2$, applied to Problem 1.}\label{mil_pol_05}
\end{figure}
\end{itemize}

\bigskip
{\em Problem 2}.  Let us consider the scalar nonlinear SDE \eqref{equation} with 
\begin{equation}
\notag
f(X(t)) = -5X(t), \quad \quad g(X(t)) = \sin(X(t))
\end{equation}
and initial data $X_0=1$ and $Y_0=0$. For this problem the constants $L$ and $\mu$ are given by $L=1$ and $\mu=-5$, so $\alpha=-9$. Then, according to Theorem \ref{thExp}, the problem generates mean-square contractive solutions. Moreover, the values of $M$ in \eqref{M} and $\widetilde{M}$ in \eqref{Mtilde} are 25 and 1, respectively. We consider the following $A$-stable methods: 

\medskip

\begin{itemize}
\item the stochastic trapezoidal methods, i.e., the $\theta$-Maruyama methods \eqref{teta} with $\theta=1/2$. In this case (\ref{cc1}) yields 
$$\mathcal{R}=\left(0, \ \frac{36}{25}\right).$$
The corresponding estimate on $\Delta t$ is confirmed in Figure \ref{mar_nonPol_05}, as well as the convergence of the numerical slope $\nu(\frac12,\Delta t)$ in \eqref{contCond} to the exact slope $\alpha$ in \eqref{exp}. Also in this case, for values of $\Delta t>\frac{36}{25}$, the mean-square deviation does not exponentially decay;
\begin{figure}
\centering
         \scalebox{0.5}{\includegraphics{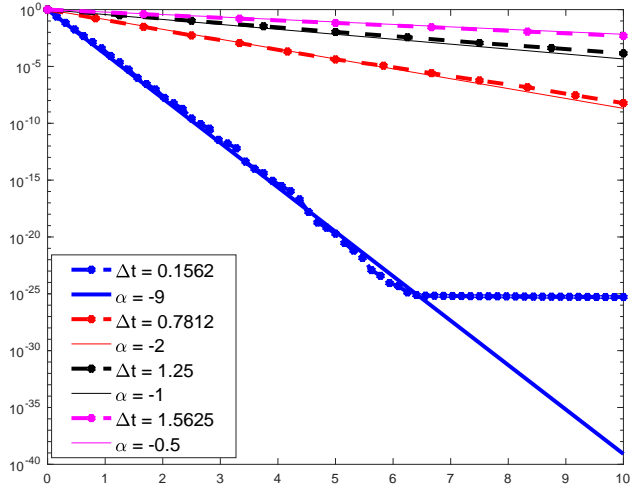}}
          \vspace{-0.4cm}\caption{Mean-square deviations over 2000 paths for the stochastic trapezoidal method applied to Problem 2.}\label{mar_nonPol_05}
\end{figure}
\item the $\theta$-Maruyama methods \eqref{teta} with $\theta=13/20$. In this case, according to (\ref{cc1}), we have  
$$\mathcal{R}=\left(0, \ \frac{144}{49}\right).$$
Also in this case the numerical evidence reported in Figure \ref{mar_nonPol_065} confirms the theoretical results;
\begin{figure}
\centering
         \scalebox{0.5}{\includegraphics{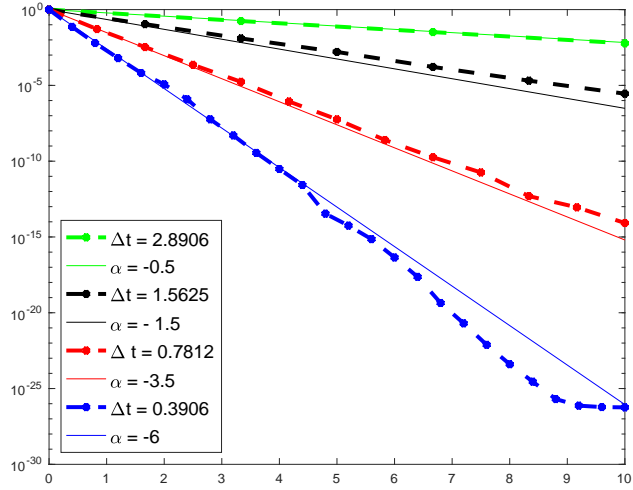}}
          \caption{Mean-square deviations over 2000 paths for the $\theta$-Maruyama method with $\theta=13/20$ applied to Problem 2.}.\label{mar_nonPol_065}
\end{figure}
\item the $\theta$-Milstein method \eqref{teta_m} with $\theta=13/20$.  For this method, (\ref{cc2}) leads to 
$$\mathcal{R}=\left(0, \ \frac{48}{19}\right).$$
The numerical evidence, confirming the theoretical results, is shown in Figure \ref{mil_nonPol_065}.
\begin{figure}
\centering
         \scalebox{0.5}{\includegraphics{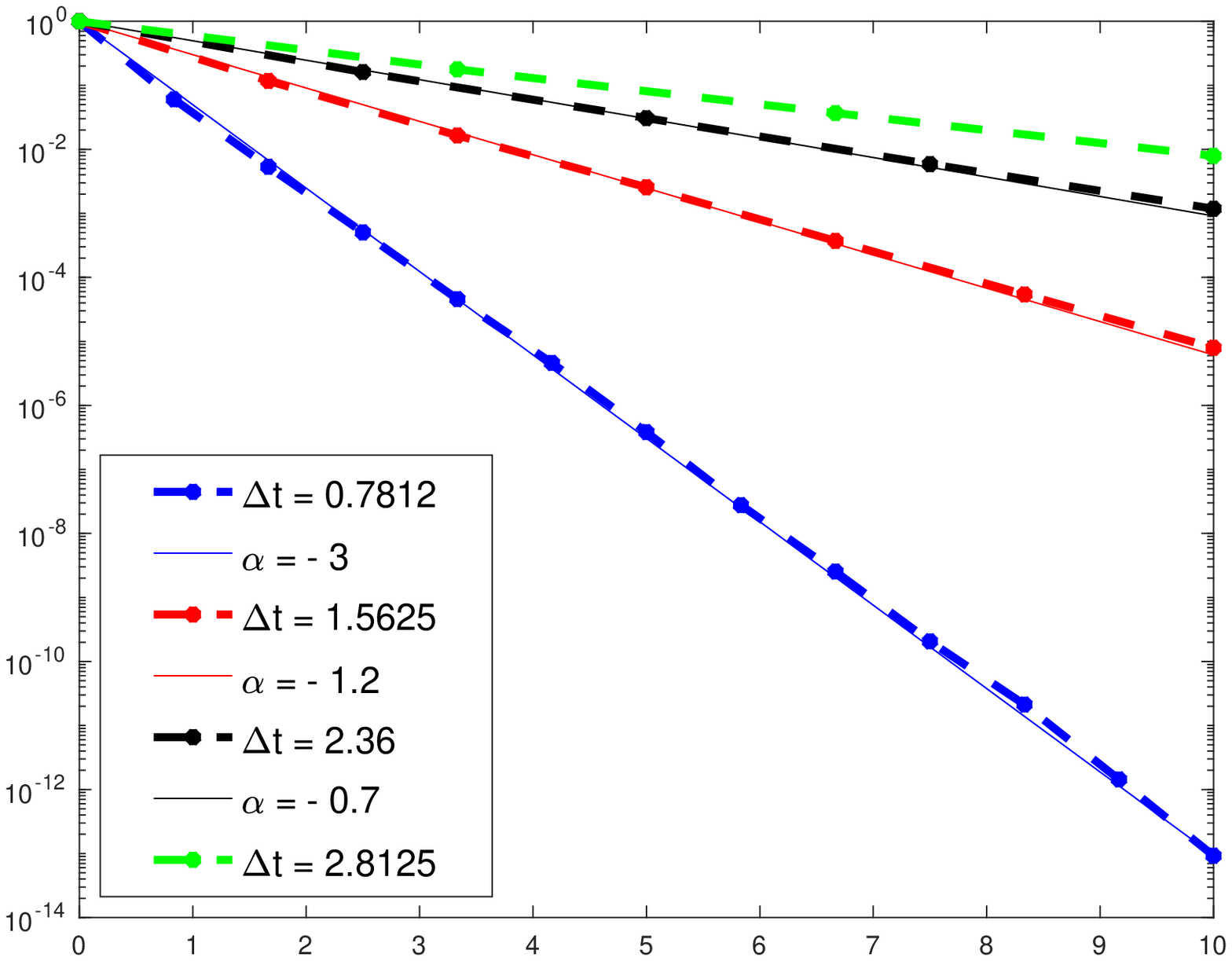}}
          \caption{Mean-square deviations over 2000 paths for the $\theta$-Milstein method \eqref{teta_m} with $\theta=13/20$, applied to Problem 2.}\label{mil_nonPol_065}
\end{figure}
\end{itemize} 
{\em Problem 3}.  We finally consider the nonlinear system of SDEs with
 \begin{equation}
 \notag
 f(X(t))= -4 \left[\begin{aligned} &\sin(X_1(t))  \\[1mm] &\sin(X_2(t))
\end{aligned}\right],
\quad \quad g(X(t)) = \frac{1}{7} \left[ \begin{array}{cc} X_1(t)  \ & \phantom{-}\displaystyle\frac{3}{2}X_2(t)  \\[5mm] \displaystyle\frac{5}{2}X_1(t) \ & -\displaystyle\frac{1}{2}X_2(t)
\end{array}\right].
 \end{equation}
and initial data $X_0=[1 \quad 1]\trx$ and $Y_0=[0 \quad 0]\trx$. For this problem the constants $L$ and $\mu$ are estimated as $L=0.148$ and $\mu= - 3.56$, so $\alpha\approx - 7.5$ and, as a consequence, the problem generates mean-square contractive solutions. Moreover, the value of $M$ in \eqref{M} is equal to 16. Also for this problem, we consider the following $A$-stable methods: 
 
 \medskip
 
\begin{itemize}
\item the stochastic trapezoidal methods, i.e., the $\theta$-Maruyama methods \eqref{teta} with $\theta=1/2$. In this case (\ref{cc1}) yields 
$$\mathcal{R}=\left(0, \ 1.1875\right).$$
This estimate is confirmed in Figure \ref{mar_sis_05}, as well as the exponential decay of the mean-square deviation with slope tending to the exact slope as $\Delta t$ decreases;
\begin{figure}
\centering
         \scalebox{0.4}{\includegraphics{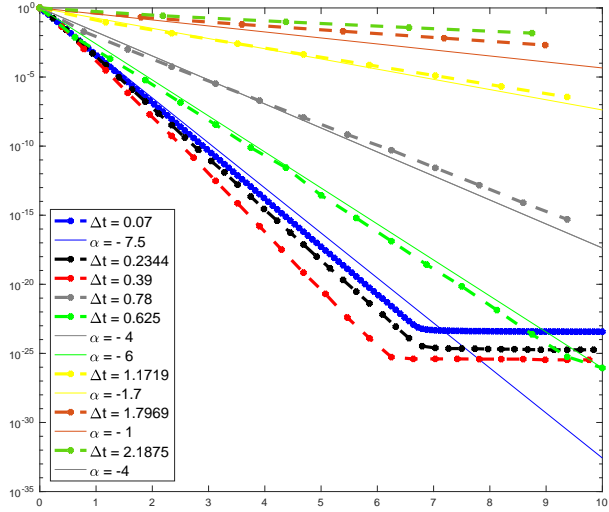}}
          \vspace{-0.4cm}\caption{Mean-square deviations over 2000 paths for the stochastic trapezoidal method applied to Problem 3.}\label{mar_sis_05}
\end{figure}
\item the stochastic implicit Euler \eqref{implEM}, whose unconditional mean-square contractivity is confirmed by the numerical evidence reported in Figure \ref{mar_sis_1}.
 \begin{figure}
\centering
         \scalebox{0.5}{\includegraphics{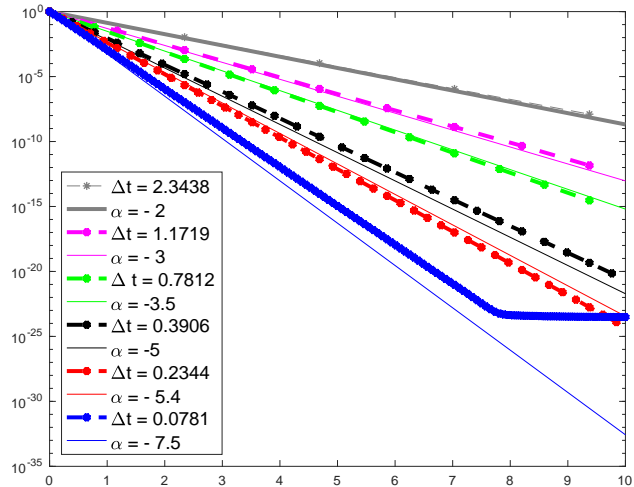}}
          \vspace{-0.4cm}\caption{Mean-square deviations over 2000 paths for the stochastic implicit Euler method \eqref{implEM}, applied to Problem 3.}\label{mar_sis_1}
\end{figure}
\end{itemize}

\section{Conclusions} 
In this paper we have analyzed featured nonlinear stability properties of the stochastic $\theta$-Maruyama \eqref{teta} and $\theta$-Milstein \eqref{teta_m} methods for nonlinear SDEs \eqref{equation} satisfying the assumptions of Theorem \ref{thExp}, hence fulfilling an exponential mean-square stability inequality of type \eqref{exp}. According to Definition \ref{def1}, if the parameter $\alpha$ in (\ref{exp}) is negative, the problem is said to generate exponential mean-square contractive solutions. We have translated this feature of the continuous problem into stepsize restrictions guaranteeing that the exponential mean-square contractive behaviour is also visible numerically. Such restrictions depend on characteristic parameters of the problem (e.g., the Lipschitz constant of the diffusion term and the one-sided Lipschitz constant of the drift in \eqref{equation}) that have been estimated through the algorithms presented in Section 4. The overall developed theory provides sharp stepsize restrictions that have also been confirmed on a selection of scalar and vector valued problems. 
Future issues of this research regard the analysis of mean-square contractivity properties for stochastic Runge-Kutta methods, eventually leading to a notion of stochastic algebraic stability, in analogy with a similar features occurirng in the deterministic case.

\end{document}